\documentclass[12pt]{amsart}

\usepackage{amsmath,amsfonts,amssymb}
\usepackage{color,graphicx}
\usepackage[left=2.5cm,right=2.5cm,top=3.5cm,bottom=3cm]{geometry}
\usepackage{xcolor}
\usepackage{amsthm}
\usepackage{ dsfont }
\newtheorem{thm}{Theorem}[section]

\newtheorem{defn}[thm]{Definition}
\newtheorem{cor}[thm]{Corollary}
\newtheorem{prop}[thm]{Proposition}

\theoremstyle{definition}
\newtheorem{rem}[thm]{Remark}

\usepackage[utf8]{inputenc}

\newcommand{\D}{\mathcal{D}}

\newcommand{\N}{\mathds{N}}

\newcommand{\R}{\mathds{R}}

\newcommand{\K}{\mathds{K}}
\newcommand{\dist}{\textup{dist}}

\newcommand{\spann}{\textup{span}}
\newcommand{\vertiii}[1]{{\left\vert\kern-0.25ex\left\vert\kern-0.25ex\left\vert #1 
    \right\vert\kern-0.25ex\right\vert\kern-0.25ex\right\vert}}

\date{\today}
\title[An Epsilon-Hypercyclicity Criterion and its application on classical Banach spaces]{An Epsilon-Hypercyclicity Criterion and its application on classical Banach spaces}
\author[S. Tapia-Garc\'ia]{Sebastián Tapia-García}
\address{Sebastián Tapia-García}
\address{ Departamento de Ingenier\'ia Matem\'atica, CMM (CNRS UMI 2807) Universidad de Chile\\ Beauchef 851, Santiago, Chile.}
\address{Institute de Math\'ematique de Bordeaux, IMB (CNRS UMR 5251) Universit\'e de Bordeaux, Course de la liberation 351, Talence, France.}
\email{stapia@dim.uchile.cl}
\begin{document}

\begin{abstract}
	We provide a criterion for $\varepsilon$-hypercyclicity. Also, complementing the ideas of Badea, Grivaux, Müller and Bayart, we construct $\varepsilon$-hypercyclic operators which are not hypercyclic in a wider class of separable Banach spaces, including several classical Banach spaces. For instance, our results can be applied to separable infinite dimensional $L^p$ spaces and $C(K)$ spaces. 
\end{abstract}

\maketitle
\noindent \textbf{MSC2020:} Primary: 47A16, 47B37.\\
\noindent \textbf{Key words and phrases:} Orbits of a linear operator, $\varepsilon$-hypercyclic operator, $\varepsilon$-hypercyclicity criterion. 
\section{Introduction}
Let $X$ be a separable infinite dimensional real or complex Banach space and $T$ be a linear bounded operator on $X$. 
For each $x\in X$, the orbit of $x$ under the action of $T$ is the set $\textup{Orb}_T(x):=\{T^nx:n\in\N\}$. 
During the last decades, the study of orbits of bounded operators has been a prolific area in Functional Analysis. 
Indeed, this is related with the following invariant subspace/subset problem: 
Does there exist a linear bounded operator $T$ without non-trivial invariant closed subspace/subset?
One of the most important properties in terms of the dynamic generated by an operator is the so-called hypercyclicity, which appeared for first time in \cite{Bi} for Fr\'echet spaces and in \cite{Ro} for Banach spaces. 
According to \cite{BM}, an operator $T$ is called hypercyclic if there exists a point $x\in X$ such that $\textup{Orb}_T(x)$ is dense in $X$. 
The vector $x$ is said to be a hypercyclic vector of $T$.
It easily follows that an operator $T$ has no non-trivial invariant closed subset if and only if each nonzero vector is hypercyclic.
The invariant subset (or subspace) problem has been solved in some particular cases, but it remains open in reflexives spaces.
In particular, this is an open problem in the separable infinite dimensional Hilbert space, see for instance \cite{G,R}.
The effort of researchers to understand linear bounded operators in terms of their dynamic has increased along the last 40 years.
Nowadays, there are several different properties which are somehow related to hypercyclicity.
Further information can be found in \cite{BM} and references therein.\\

Let us write down the principal object of our research, which was introduced for first time in \cite{BGM}.
\begin{defn}
	Let $X$ be a Banach space and let $\varepsilon>0$. 
	A linear bounded operator $T$ on $X$ is called $\varepsilon$-hypercyclic if there exists a vector $x$ such that for all $y\in X\setminus\{0\}$, there exists $n\in \N$ for which
	\[\|T^nx-y\| \leq \varepsilon\|y\|.\]
	The vector $x$ is said to be an $\varepsilon$-hypercyclic vector of $T$.
\end{defn}
By the very definition, it is clear that every hypercyclic operator is $\varepsilon$-hypercyclic for every $\varepsilon>0$. 
So, we already know that $\varepsilon$-hypercyclic operators exist in every separable infinite dimensional Banach space, see \cite{A} and \cite{Be}. 
Also, every linear operator is $1$-hypercyclic, since $0$ is a $1$-hypercyclic vector.
However, up to the best of our knowledge, it remains open  the question if every separable infinite dimensional Banach space admits an $\varepsilon$-hypercyclic which is not hypercyclic (with $\varepsilon\in (0,1)$).
Up to the best of our knowledge, in the literature we can find the construction of such an operator in the spaces $\ell^1(\N)$ and $\ell^2(\N)$, see \cite{BGM} and \cite{B} respectively.
We point out that the technique used by Bayart in \cite{B} is motivated by the one used by Badea, Grivaux and Müller in \cite{BGM}. 
Also, we can find other mentions of $\varepsilon$-hypercyclicity in \cite{BC,BBF,P}.\\

In this work, we introduce the following $\varepsilon$-Hypercyclicity criterion (Theorem \ref{criterion}) and we use it to provide several examples of $\varepsilon$-hypercyclic operators which are not hypercyclic (Theorem \ref{product space corollary}, Corollary \ref{Lp spaces}, Corollary \ref{C(K) space}).

\begin{thm}[$\varepsilon$-Hypercyclicity Criterion]\label{criterion}
	Let $X$ be a separable Banach space, let $T$ be a bounded operator on $X$ and let $\varepsilon\in(0,1)$. Let $\D_1$ be a dense set on $X$. Let $\D_2$ be a countable subset of $X$. Let us fix an enumeration of $\D_2:=\{y_k:k\in\N\}$. Assume further that 
	for each $x\in X\setminus\{0\}$, there are infinitely many integers $k\in\N$ such that $y_k\in B(x,\varepsilon\|x\|)$. 
	Let $(n(k))_k\subset \N$ be an increasing sequence and let $S_{n(k)}:\D_2\to X$ be a sequence of maps such that:
	\begin{enumerate}
		\item $\lim_{k\to\infty}\|T^{n(k)}x\|= 0$ for all $x\in \D_1$,
		\item $ \lim_{k\to\infty}\|S_{n(k)}y_k\|=0$,
		\item $ \lim_{k\to\infty}\|T^{n(k)}S_{n(k)}y_k-y_k\| = 0$.
	\end{enumerate}
	Then, $T$ is $\delta$-hypercyclic for all $\delta >\varepsilon$.
\end{thm}

Provided with the preceding criterion, we construct $\varepsilon$-hypercyclic operators which are not hypercyclic in more general separable infinite dimensional Banach spaces. 
For instance, separable infinite dimensional $L^p$ spaces, with $p\in[1,\infty)$ and $C(K)$ spaces are enclosed by our result. Concretely, a consequence of our results reads as follows

\begin{thm}\label{product space corollary}
	Let $X$ be a separable Banach space. Assume that $X$ contains a complemented subspace isomorphic to $c_0(\N)$ or $\ell^p(\N)$, for $p\in[1,\infty)$. Then $X$ admits an $\varepsilon$-hypercyclic operator which is not hypercyclic.
\end{thm}

The paper is organized as follows. 
In the next section, we provide all necessary definitions to carry out our results.
In the third section we recall the Hypercyclicity Criterion and we prove our $\varepsilon$-Hypercyclicity Criterion. 
Also, we present a sufficient condition to ensure that the direct sum of a hypercyclic operator and an $\varepsilon$-hypercyclic operator remains $\varepsilon$-hypercyclic in the product space.
The fourth section is devoted to prove an abstract version of Theorem \ref{product space corollary}.
In the fifth section, we discuss why a natural choice for an $\varepsilon$-Hypercyclicity Criterion implies that the operator satisfies the Hypercyclicity Criterion.
Finally, we end this manuscript with some proofs of simple but useful facts used previously through this work.

\section{Notation and Preliminaries}
Let us start fixing the notation used in this work. We denote by $X$ and $Y$ infinite dimensional real or complex separable Banach spaces. 
By $X^*$ we denote the dual space of $X$.
For $x\in X$ and $r>0$, $B(x,r)$ denotes for the closed ball centered at $x$ with radius $r$.
By operator on $X$ we mean a linear bounded map from $X$ to $X$.
A sequence $(e_n,e^*_n)_n\subset X\times X^*$ is called a biorthogonal system if $e_n^*(e_m)=\delta_{n,m}$, where $\delta_{n,m}$ stands for the Kronecker's symbol. 
\begin{defn} Let $X$ be a Banach space. We say that a biorthogal system $(e_n,e^*_n)_n\subset X\times X^*$ is a bounded M-basis on $X$ if $X=\overline{\spann}(e_n:n)$, $X^*=\overline{\spann}^{\omega^*}(e^*_n:n)$ and $\sup_{n\in\N} \|e_n\|\|e_n^*\|<\infty$.
\end{defn}
A classical result of Ovsepian and Pe\l czy\'nki asserts that each separable Banach space admits a bounded M-basis, see \cite{OP}.
\begin{defn} Let $X$ be a Banach space. We say that $(e_n)_n\subset X$ is a basis of $X$ if, for each vector $x$, there exists a unique sequence $(x_n)_n\subset \K$ such that $x=\sum_{n=1}^{\infty} x_ne_n$. Its associated biorthogonal system, denoted by $(e_n^*)_n$, is the sequence of linear bounded forms on $X$ defined by $e_n^*(x)=x_n$. 
\end{defn}

\begin{defn}
	Let $(e_n)_n$ be a basis on $X$. We say that $(e_n)_n$ is $C$-unconditional, for $C\geq 1$, if for any sequences of scalars $(a_n)_n,(b_n)_n\subset \K$ such that $|b_n|\leq |a_n|$ for all $n\in\N$, the following holds:
	\[\left\| \sum_{n=0}^{m-1}b_ne_{n_k}\right\| \leq C \left\| \sum_{n=0}^{m-1}a_ne_n\right\|,~ \forall m\in \N.\]
	
	We say that $(e_n)_n$ is unconditional if it is $C$-unconditional for some $C\geq 1$.
\end{defn}

The following space is constructed as a generalization of the spaces $c_0(X)$ or $\ell^p(X)$, for $p\in[1,\infty)$.

\begin{defn}\label{bigplusYX}
	Let $X$ and $Y$ be two Banach spaces. Suppose that $(f_n)_n\subset Y$ is a normalized $1$-unconditional basis for $Y$. We denote by $\bigoplus_Y X$ the vector space defined by
	\[\bigoplus_Y X:=\{(x_n)_n\in X^\N:~  \sum_{n=0}^\infty \|x_n\|_X f_n\in Y\}.\]
	We endow this space with the norm $\|\cdot\|$ defined by $\|(x_n)_n\|=\| \sum_{n=0}^\infty \|x_n\|_X f_n\|_Y$.
\end{defn}
Observe that, since the basis $(f_n)_n$ is $1$-unconditional, the proposed norm for $\bigoplus_Y X$ trivially satisfies triangle inequality.
A standard argument shows that $\bigoplus_Y X$ is a Banach space.
Clearly, the space constructed in Definition \ref{bigplusYX} depends on the chosen basis $(f_n)_n$ of $Y$, but we omit it for sake of brevity. 
Observe that, if $X$ is either $c_0(\N)$ or $\ell^p(\N)$, for $p\in [1,\infty)$, then $X$ is isometric to $\bigoplus_XX$, whenever we use the canonical basis of $X$.
Also, notice that for all $(x_n)_n\in\bigoplus_Y X$, the sequence $(\|x_n\|_X)_n$ tends to $0$.

\section{Epsilon-hypercyclicity criterion}

In the literature we can find sufficient conditions to prove that a given operator has a vector whose orbit satisfies some property. 
For instance, the Hypercyclicity Criterion or the Supercyclicity Criterion. 
We start this section recalling the former one.
However, there are hypercyclic operators which does not satisfy this criterion. For instance, see \cite{RR}.

\begin{thm}[Hypercyclicity Criterion]\label{hypercyclic criterion}
	Let $X$ be an infinite dimensional separable Banach space and let $T$ be a bounded operator on $X$. If there exist a sequence of integers $(n(k))_k$, two dense sets $\D_1$, $\D_2\subset X$ and a sequence of maps $S_{n(k)}:\D_2\to X$ such that:
	\begin{enumerate}
		\item $T^{n(k)}x\to 0 $ for all $x\in \D_1$.
		\item $S_{n(k)}y\to 0 $ for each $y\in \D_2$.
		\item $T^{n(k)}S_{n(k)}y\to y$ for each $y\in \D_2$.
	\end{enumerate}
	Then $T$ is hypercyclic.
\end{thm}
Now, we continue with the proof of Theorem \ref{criterion}.

\begin{proof}[Proof of the $\varepsilon$-Hypercyclicity Criterion]
	Let us construct a $\delta$-hypercyclic vector of $T$, for any $\delta>\epsilon$.
Let $(\eta_k)_k\subset \R^+$ be any sequence of positive numbers such that $(k^2\eta_k)_k$ converges to $0$.
Observe that the series $\sum_k\eta_k$ is convergent.
Let $\{z_k:k\in\N\}$ be a countable dense subset of $X$. 
Let $m_0\in\N$ such that 

\[ \|z_0-y_{m_0}\|\leq\varepsilon\|z_0\|,~ \|S_{n(m_0)}y_{m_0}\|<\eta_0 ~\text{and } \|T^{n(m_0)}S_{n(m_0)}y_{m_0}-y_{m_0}\|<\eta_0. \] 
By density of $\D_1$ and continuity of $T$, there is $x_0\in \D_1$ such that $\|x_0\|<\eta_0$ and $\|T^{m_0}x_0-y_{m_0}\|<\eta_0$.
Let $k\geq 1$ and let us assume that $(x_i)_i\subset \D_1$ and $(m_i)_i\subset \N$ are already defined for all $i\leq k-1$. 
Let $\rho_k>0$ be a positive number such that $\|T^{n(m_i)}u\|\leq 2^{-k}$ for all $\|u\|\leq \rho_k$ and for all $i<k$. 
Redefine $\eta_k:= \min\{\eta_k,\rho_k\}$.
Let $m_k$ be an integer such that $m_k>m_{k-1}$, $\|T^{n(m_k)}x_i\| < \eta_k$ for all $i<k$,

\[\|z_k-y_{m_k}\|\leq \varepsilon\|z_k\|,~ \|S_{n(m_k)}y_{m_k}\|< \eta_k ~\text{and } \|T^{n(m_k)}S_{n(m_k)}y_{m_k}-y_{m_k}\|<\eta_k.\]

By density of $\D_1$, there is $x_k\in \D_1$ such that $\|x_k\| < \eta_k$ and $\|T^{n(m_k)}x_k-y_{m_k}\|<\eta_k$. \\

Now, since $\|x_k\|<\eta_k$ for all $k\in \N$, the vector $\overline{x}=\sum_{k=0}^\infty x_k\in X$ is well defined. 
We claim that $\overline{x}$ is a $\delta$-hypercyclic vector for $T$, for all $\delta>\varepsilon$. Indeed, let $j\in \N$. Then:

\begin{align*}
	\|T^{n(m_j)}\overline{x}-z_j\| &\leq \sum_{k=0}^{j-1}\|T^{n(m_j)}x_k\| + \|T^{n(m_j)}x_j-y_{m_j}\|+\|y_{m_j}-z_j\|+ \sum_{k=j+1}^\infty\|T^{n(m_j)}x_k\|\\
	&\leq (j+1)\eta_j + \varepsilon\|z_j\|+ \sum_{k=j+1}^\infty 2^{-k}.
\end{align*}
\noindent Now, let $z\in X$ and let $(j(l))_l$ be an increasing sequence such that $(z_{j(l)})_l$ converges to $z$. Then

\begin{align*}
	\|T^{n(m_{j(l)})}\overline{x}-z\| &\leq \|T^{n(m_{j(l)})}\overline{x}-z_{j(l)}\|+\|z_{j(l)}-z\| \\
	&\leq (j(l)+1)\eta_{j(l)} + \varepsilon\|z_{j(l)}\|+ \sum_{k={j(l)+1}}^\infty 2^{-k}+\|z_{j(l)}-z\|,
\end{align*}

\noindent expression which tends to $\varepsilon\|z\|$ as $l$ tends to infinity. 
Therefore, if $\delta>\varepsilon$ and $z\neq 0$, for $l$ large enough, we have that $\|T^{n(m_{j(l)})}\overline{x}-z\| \leq \delta \|z\|$.
\end{proof}

\begin{rem}\label{remark limsup}
   From the proof, notice that if the sequence $(z_{j(l)})_l$ converges to $z$, then
   \[\limsup _{l\to\infty} \|T^{n(m_{j(l)})}\overline{x}-z\|\leq \varepsilon\|z\|.\]
\end{rem}
\begin{rem}
	The previous criterion can be applied to the operators constructed in \cite{BGM} and \cite{B}. In fact, this can be done similarly as we do in Theorem \ref{ehypercyclic2}.
\end{rem}

By definition, every hypercyclic operator is $\varepsilon$-hypercyclic for all $\varepsilon>0$. In \cite{BGM} it is shown that the converse is also true, 
i.e. if an operator is $\varepsilon$-hypercyclic for each $\varepsilon>0$, then it is hypercyclic.
In this line, we have the following result.
\begin{prop}\label{hyper implies ehyper}
	Let $X$ be a separable infinite dimensional Banach space and let $T$ be a bounded operator on $X$. If $T$ satisfies the Hypercyclicity Criterion, then it satisfies  the $\varepsilon$-Hypercyclicity Criterion for every $\varepsilon>0$.
\end{prop}

\begin{proof}
	Let $T$ be a bounded operator satisfying the Hypercyclicity Criterion.
Let $\D_1,\D_2\subset X$, $(n(k))_k\subset\N$ and $(S_{n(k)})_k$ given by the mentioned criterion. 
Since $X$ is separable, without loss of generality, we can assume that $\D_2$ is a countable dense set. 
Let us enumerate ${\D_2:=\{y_k:k\in\N\}}$. 
To achieve the $\varepsilon$-Hypercyclic-Criterion we only need to construct a subsequence of $(n(k))_k$, namely $(m(k))_k$, which satisfies hypothesis $(2)$ and $(3)$ of Theorem \ref{criterion}. 
To this end, let us define $m(0)\in \{n(k):k\in\N\}$ such that $\|S_{m(0)}y_0\| \leq 1$ and $\|T^{m(0)}S_{m(0)}y_0- y_0\|\leq 1$.
Let $k\geq 1$ and suppose that we have constructed $(m(j))_j$ for all $j\leq k-1$. 
Let us fix ${m(k)\in  \{n(j):j\in\N\}}$ such that

\[m(k)>m(k-1),~\hspace{0.5cm}\|S_{m(k)}y_k\| \leq k^{-1},~\hspace{0.5cm}~\text{ and }\hspace{0.5cm} \|T^{m(k)}S_{m(k)}y_k- y_k\|\leq k^{-1}.\] 

Now, it is straightforward that hypothesis $(1)$, $(2)$ and $(3)$ of the $\varepsilon$-Hypercyclicity Criterion are satisfied for the sequence $(m(k))_k$ and the maps $(S_{m(k)})_k$. 
Finally, since $\D_2$ is dense, the intersection of $\D_2$ with any open set must be an infinite set.
Therefore, $T$ satisfies the $\varepsilon$-Hypercyclicity Criterion for each $\varepsilon>0$.
\end{proof}

Proposition \ref{product space}, combined with Theorem \ref{ehypercyclic2}, allows us to construct $\varepsilon$-hypercyclic operators which are not hypercyclic in several classical spaces. 
Bearing Proposition \ref{hyper implies ehyper} in mind, this result can be seen as a generalization of the necessity part of the following Theorem of Bès and Peris: $T$ satisfies the Hypercyclicity Criterion if and only if $T\oplus T$ is hypercyclic. See \cite[Theorem 2.3]{BP}.
\begin{prop}\label{product space}
	Let $X$ and $Y$ be two separable infinite dimensional Banach spaces and let $\varepsilon\in(0,1)$. 
	Let $T\in\mathcal{L}(X)$ satisfying the Hypercyclicity Criterion.
	Let $S\in\mathcal{L}(Y)$ satisfying the $\varepsilon$-Hypercyclicity Criterion. 
	Further, assume that the sequences of integers provided by both criteria are the same. 
	Then, the operator $T\oplus S$ is $\delta$-hypercyclic on $X\oplus Y$, for all $\delta >\varepsilon$, where $X\oplus Y$ is equipped with the norm of the maximum.
\end{prop}

Before proceeding with the proof of Proposition \ref{product space}, we recall the following simple fact: If $T$ satisfies the Hypercyclicity Criterion under some sequence $(n(k))_k$, then, it will satisfies the criterion for any subsequence $(n(k(j)))_j$.

\begin{proof}

	Let $(n(k))_k$ be an increasing sequence of integers, let $\D^X_1,\D^X_2\subset X$ be two dense sets and let $(U_{n(k)})_k$ a sequence of maps, provided by the Hypercyclicity Criterion for $T$.
Let $\D^Y_1\subset Y$ be a dense set, let $\D^Y_2:=\{z_k:k\in \N\}\subset Y$ and let $(V_{n(k)})_k$ be a sequence of maps provided by the $\varepsilon$-Hypercyclicity Criterion for $S$, all of them related to the sequence $(n(k))_k$. \\

\noindent Let $(v_k)_k \subset Y\setminus \{0\}$ be a dense sequence such that $v_i\neq v_j$ if $i\neq j$. 
Summarizing the constructive proof of Theorem \ref{criterion}, we can obtain a subsequence of $(z_k)_k$, which we still denote by $(z_k)_k$, a sequence $(y_k)_k\subset \D_1^Y$ and a fast decreasing null sequence $(\eta_k)_k\subset\R^+$ such that:
\begin{itemize}
	\item $\|v_k-z_k\|\leq \varepsilon\|v_k\|$, for all $k\in\N$,
	\item $\|y_k\|\leq \eta_k$  for all $k\in \N$,
	\item $\|S^{n(k)}y_i\|\leq \eta_k$ for all $i<k$,
	%\item $\|V_{n(k)}z_k\|\leq \eta_k$ algo for all $k\in\N$,
	\item $\|S^{n(k)}y_k-z_k\|\leq \eta_k$ for all $k\in\N$.
	
\end{itemize}
Further, the vector $y=\sum_k y_k$ is $\delta$-Hypercyclic, for all $\delta>\varepsilon$. 
For each $i\in\N$, let us consider an increasing sequence $(k(i,j))_j$ such that $v_{k(i,j)}$ converges to $v_i$, as $j$ tends to infinity. 
Inductively, we define the sets $\N_0:=\{k(0,j):j\in\N\}$ and, for $i\geq 1$, $\N_i:=\{k(i,j):j\in\N\}\setminus \bigcup_{k<i} \N_k$. 
Since the sequence $(v_k)_k$ is injective, the sets $\N_i$ are infinite for each $i\in\N$. 
Observe that, by Remark~	\ref{remark limsup}, the expression
\begin{align}\label{second coordinate}
	\limsup_{t\in \N_k,~t\to \infty}\|S^{n(t)}y-v_k\|\leq \varepsilon\|v_k\|,  
\end{align}
holds true for each $k\in\N$. \\

\noindent Now, let us construct a hypercyclic vector $x$ of $T$ adapted to $y$ in the following sense: the set $\{T^nx:n\in\N_j\}$ is dense in $X$ for each $j\in\N$.
Assume that $\D^X_2$ is a countable set and fix an enumeration of it, i.e. $\D^X_2=\{w_l:l\in\N\}$. 
Let us define the following total order on $\N^2$: for $(i,j),(k,l)\in \N^2$, we write 
\[(i,j)\preceq (k,l) ~~~~\text{if } \hspace{0.5cm} i+j< k+l ~\hspace{0.2cm}  \text{or } \hspace{0.2cm} i+j= k+l \wedge i\geq k.\]
Let $m(0,0)\in \N_0$ and $x_{0,0}\in \D_1^X$ such that $\|U_{m(0,0)}w_0\|\leq 1$, $\|x_{0,0}\|\leq 1$ and ${\|T^{m(0,0)}x_0-w_0\|\leq 1}$. 
Now, let us proceed by induction. Let $k,l\in\N$.  
Suppose that we have constructed $m(i,j)$ and $x_{i,j}$ for all $(i,j)\prec (k,l)$. Let $m(k,l)\in \N_k$ and $x_{k,l}\in \D_1^X$ such that 
\begin{itemize}
	\item $m(k,l)>m(i,j)$ for all $(i,j)\prec (k,l)$.
	\item $\|T^{m(k,l)}x_{i,j}\|\leq \rho(k,l)$, for all $(i,j)\prec (k,l)$,
	\item $\|T^{m(i,j)}x_{k,l}\|\leq 2^{-k-l}$, for all $(i,j)\prec (k,l)$,
	\item $\|U_{m(k,l)}w_l\|< \rho(k,l)$,
	\item $\|x_{(k,l)}\|\leq \rho(k,l)$,
	\item $\|T^{m(k,l)}x_{(k,l)}-w_l\|\leq \rho(k,l)$,  
\end{itemize}
where $\rho:\N^2\to\R^+$ is a decreasing function (for $\preceq$) such that $(k+l)^3\rho(k,l)$ tends to $0$ whenever $(k,l)$ tends to infinity through the order $\preceq$. 
Thus, we claim that the vector $x=\sum_{i,j\geq 0}x_{i,j}$ is well defined and a hypercyclic vector of $T$. 
Moreover, for each $i\in\N$, the set 
\begin{align}\label{first coordinate}
	\{T^{m(i,j)}x:j\in\N\}~\text{is dense in }X. 
\end{align}
Indeed, the claim follows from the next computation and the fact that $(w_l)_l$ is dense in $X$. Let $(k,l)\in\N^2$. Then we get
\begin{align*}
	\|T^{m(k,l)}x-w_l \|&\leq \sum_{(i,j)\prec (k,l)} \|T^{m(k,l)}x_{i,j} \| +\|T^{m(k,l)}x_{k,l}-w_l\| + \sum_{(k,l)\prec (i,j)}\|T^{m(k,l)}x_{i,j}\|\\ 
	&\leq \rho(k,l)\left(\dfrac{(k+l+1)(k+l+2)}{2}+1\right) +\sum_{(k,l)\prec (i,j)}2^{-i-j},
\end{align*}
where the last expression tends to $0$ as $k$ tends to infinity. 
Observe that, by construction, the sequence $(m(i,j))_j\subset \N_i$ for all $i\in\N$.\\

Let us equip the product space $X\oplus Y$ with the norm of the maximum, i.e. $\|(a,b)\|=\max\{\|a\|_X,\|b\|_Y\}$, for all $(a,b)\in X\oplus Y$.
Let us prove that the vector $(x,y)$ is $\delta$-hypercyclic for $T\oplus S$, for all $\delta>\varepsilon$.
Combining \eqref{second coordinate} and \eqref{first coordinate}, we obtain that
\[\inf_{j\in \N}\|(T\oplus S)^{m(i,j)}(x,y)-(a,v_k)\|\leq \varepsilon\|(0,v_k)\|\leq \varepsilon\|(a,v_k)\|,~ \forall~a\in X,\forall~k\in\N.\]
Let $(a,b)\in X\oplus Y\neq (0,0)$, using triangle inequality and the previous inequality we get
\[\inf_{n\in\N}\|(T\oplus S)^{n}(x,y)-(a,b)\|\leq \inf_{k\in\N}\varepsilon\|(a,v_k)\| +\|(a,b)-(a,v_k)\|\leq \varepsilon \|(a,b)\|. \]
Since $(v_k)_k$ is dense in $Y$ and $(a,b)\neq (0,0)$, by definition of infimum we finally conclude that $(x,y)$ is $\delta$-hypercyclic for each $\delta >\varepsilon$. 
	
\end{proof}

\section{Construction of epsilon-hypercyclic operators}\label{construction2}

In this section we prove Theorem \ref{Theorem product spaces}, which is a technical version of Theorem \ref{product space corollary}. We start with some applications of Theorem \ref{product space corollary}.

\begin{cor}\label{ehypercyclic corollary2}
	The following Banach spaces admit $\varepsilon$-hypercyclic operator which are not hypercyclic: %Orlicz spaces, 
	$\ell^p(X)$ for $p\in [1,+\infty)$ and $c_0(X)$ whenever $X$ is a (finite or infinite dimensional) separable Banach space.
\end{cor}

\begin{proof} 
	It is a direct consequence of Theorem \ref{product space corollary}.
\end{proof}
\begin{cor}\label{Lp spaces}
	Let $X$ be a separable infinite dimensional $L^p$ space, with $p\in[1,\infty)$. Then $X$ admits an $\varepsilon$-hypercyclic which is not hypercyclic.
\end{cor}
\begin{proof}
	Any separable infinite dimensional $L^p$ space admits a complementable subspace isomorphic to $\ell^p(\N)$. Therefore, we can apply Theorem \ref{product space corollary}.
\end{proof}
For the next corollary we recall the following classical result of Sobzyck \cite{S}:
Let $X$ be separable Banach space and $E$ be closed subspace of $X$. Let $T:E\to c_0(\N)$ be a bounded operator. Then there exists a bounded operator $\tilde{T}:X\to c_0(\N)$ such that $\tilde{T}|_E=T$. 
We denote by $C(K)$ the Banach space of continuous functions on the compact space $K$. This space is endowed with the norm of the maximum. 
\begin{cor}\label{C(K) space}
	Let $X$ be a separable infinite dimensional containing $c_0(\N)$. Then $X$ admits an $\varepsilon$-hypercyclic which is not hypercyclic. Particularly, all separable infinite dimensional $C(K)$ spaces enjoy this property.
\end{cor}

\begin{proof}
	By Sobzyck Theorem, $c_0(\N)$ is complementable on $X$. Therefore, by Theorem \ref{product space corollary}, $X$ admits an $\varepsilon$-hypercyclic which is not hypercyclic.
	Now, let us assume that $X$ is a separable infinite dimensional $C(K)$. Then, $K$ must be an infinite metrizable compact set. Hence, $X$ admits a complemented subspace isometric to $c_0(\N)$. For details see \cite[Proposition 4.3.11]{AK}.
\end{proof}

To prove Theorem \ref{product space corollary}, we show first the existence of $\varepsilon$-hypercyclic vector which are not hypercyclic in a particular class of Banach spaces, using mainly Bayart's ideas, \cite{B}.

\begin{thm}\label{ehypercyclic2}
	Let $X$ and $Y$ be two infinite dimensional separable Banach spaces. Assume that $Y$ admits a $1$-unconditional basis $(f_n)_n$ such that the associated backward shift operator is continuous. Then, for every $\varepsilon>0$, the space $\bigoplus_Y X$, related to $(f_n)$, admits an $\varepsilon$-hypercyclic operator which is not hypercyclic. 

\end{thm}

The first part of the proof consists in the formal construction of the $\varepsilon$-hypercyclic operator, whereas in the second part we prove that our operator satisfies the statement of Theorem \ref{ehypercyclic2}.
As a remark, in the second part, step 3, we apply our $\varepsilon$-Hypercyclicity Criterion.\\

Let us start the proof of Theorem \ref{ehypercyclic2}.\\

\textbf{First part}: Let $(e_n)_n\subset X$ be a normalized bounded M-basis given by the classical result of Ovsepian and Pe\l czyński, and let $(e_n^*)_n\subset X^*$ be the associated coordinates sequence.
Let $b=\sup_n\|e^*_n\|<\infty$.
Let $\varepsilon\in (0,1)$, let $\alpha>1$ and let $d\in \N$, with $d>1$, such that $2\alpha^{-d}b< \varepsilon$. 
Let $(\Delta_k)_k\subset \N$ be a rapidly increasing sequence which will be specified later on, in Proposition \ref{sequences x and z2}.
Let $(n_k)_k$ and $(n'_k)_k$ be two increasing sequences defined by $n_0=n'_0=0$, $n_k=n'_{k-1}+d+1+\Delta_k$ and $n'_k=n_k+d+1+\Delta_k$, for all $k\geq 1$. 
It is clear that $k\leq n'_{k-1}$ for all $k\geq 2$.
For $k\in \N$ and $\sigma,\beta\in\K$, we define the diagonal operator, with respect to $(e_n)_n$, $D_{k,\sigma,\beta}$ on $X$ by:
\[D_{k,\sigma,\beta}=\sigma Id+(1-\sigma)e_0^*\otimes e_0+ (\beta-\sigma)e_{k^2}^*\otimes e_{k^2}.\]
Since $D_{k,\sigma,\beta}$ is a rank $2$ perturbation of $\sigma Id$, it is a bounded operator with norm \[\|D_{k,\sigma,\beta}\|\leq |\sigma|(1+2b)+|\beta|b+b.\]
Moreover, whenever $\sigma$ and $\beta$ are different from $0$ and $k\geq 1$, $D_{k,\sigma,\beta}^{-1}=D_{k,\sigma^{-1},\beta^{-1}}$ easily follows.
For each $k\in \N$, we define the operator $N_k:=e_{k^2}^*\otimes e_0$,
i.e. $N_k(x)=e^*_{k^2}(x)e_0$ for all $x\in X$. 
Notice that $(\|N_k\|)_k$ is uniformly bounded. 
Indeed, $\|N_k\|\leq b$, for all $k\in\N$.
Also, for each $j\geq1$, we define the operator $S_j$ on $X$ as follows. 
Let $k\in \N$ be the unique integer such that $n'_{k-1}<j\leq n'_k$, then we set:

\[S_j:=\begin{cases} D_{k,\frac{1}{\alpha},\alpha}& n'_{k-1}+1\leq j\leq n'_{k-1}+d,\\
	D_{k,\frac{1}{\alpha},1}-N_k& j= n'_{k-1}+d+1,\\
	D_{k,\frac{1}{\alpha},\frac{1}{\alpha}}& n'_{k-1}+d+2\leq j \leq n'_{k-1}+d+1+\Delta_k=n_k,\\
	D_{k,\frac{1}{\alpha},\alpha}& n_{k}+1\leq j\leq n_{k}+\Delta_k,\\
	D_{k,\frac{1}{\alpha},1}+N_k& j= n_{k}+\Delta_k+1,\\
	D_{k,\frac{1}{\alpha},\frac{1}{\alpha}}& n_{k}+\Delta_k+2\leq j \leq n_{k}+d+\Delta_k,\\
	D_{k,\alpha^{n'_k-n'_{k-1}-1},\frac{1}{\alpha}}& j=n_{k}+d+1+\Delta_k=n'_{k}.
\end{cases}\]

Notice that each $S_j$ is an upper-triangular operator on $X$ with respect to the sequence $(e_n)_n$. 
Further, observe that $(D_{k,\frac{1}{\alpha},1}\pm N_k)^{-1}=D_{k,{\alpha},1}\mp N_k$. 
The following three properties are direct from the definition of the operators $S_j$.
\begin{enumerate}
	\item[($\mathcal{Q}_0$)] $S_je_0=e_0$ for all $j\geq 1$.
	\item[($\mathcal{Q}_1$)] $S_{n'_k}\cdots S_1=Id$ for all $k\geq 1$.
	\item[($\mathcal{Q}_2$)] For $k\geq 1$, $p\notin \{ 0,k^2\}$ and $i\in \{n'_{k},\cdot\cdot\cdot,n'_{k+1}-1\}$, $S^{-1}_1\cdots S^{-1}_ie_p= \alpha^{i-n'_{k}}e_p$ holds.
\end{enumerate}
Let us now formally define the operator $T$ on $\bigoplus_YX$. Let $z=(x_n)_n\in \bigoplus_YX$, then:
\[Tz=(S^{-1}_1x_1,S^{-1}_2x_2,\cdot\cdot\cdot),\]
i.e., $T$ is a backward shift on $\bigoplus_YX$ with weights $(S^{-1}_n)_n$.

\textbf{Second part:} \textit{Step 1: $T$ is a well-defined, bounded operator on $\bigoplus_YX$.} 
\begin{prop}\label{bounded inverses2}
	Each operator $S_j$ is bounded and invertible. 
	Moreover, the sequence $(\|S_j^{-1}\|)_j$ is uniformly bounded.
\end{prop}

\begin{proof}
	Let $j\geq 1$. Assume that there are $k\in \N$ and $\sigma, \beta\in \R$ such that $S_j=D_{k,\sigma,\beta}$. Recalling that $S_j^{-1}= D_{k,\sigma^{-1},\beta^{-1}}$, we get that 
	$\|S^{-1}_j\|\leq |\sigma^{-1}|(1+2b)+|\beta^{-1}|b+b$. 
	Since $\sigma\in \{\alpha^{-1},\alpha^{n'_k-n'_{k-1}-1}\}$ and $\beta\in \{\alpha^{-1},1,\alpha\}$, we conclude that $\|S^{-1}_j\|\leq \alpha(1+3b)+b$, which is a constant independent of $j$.
	Otherwise, if $S_j=D_{k,\alpha^{-1},1}\pm N_k$, then  $S_j^{-1}=D_{k,\alpha,1}\mp N_k$. 
	Therefore, $\|S_j^{-1}\|\leq \|D_{k,\alpha,1}\|+b\leq \alpha(1+2b)+2b$, which is a constant independent of $j$ as well.
\end{proof}
Let $(f_n)_n$ be the $1$-unconditional basis on $Y$ used to construct the space $\bigoplus_YX$. 
Thanks to Proposition \ref{bounded inverses2}, we know that there exists a constant $C>0$ such that $\|S_n^{-1}x\|\leq C\|x\|$ for all $x\in X$ and for all $n\in \N$.
Let $K>0$ be the norm of the Backward shift operator associated to the basis $(f_n)_n$. Then, for $z=(x_n)_n\in \bigoplus_YX$ we get
\[\|Tz\| = \left\| \sum_{n=1}^\infty\|S_n^{-1}x_n\|_X f_{n-1}\right\|_Y\leq K\left\| \sum_{n=0}^\infty C\|x_n\|_X f_n\right\|_Y= KC\|z\|,\]
which implies the well definition and continuity of $T$. \\

\textit{Step 2: $T$ is not a hypercyclic operator.} 

\begin{prop}\label{bounded products2}
	The sequence $(\|S_jS_{j-1}\cdots S_1\|)_j$ is bounded by a constant $M(d)$ which depends only on $d$.
\end{prop}

\begin{proof}
	Let $j\geq 1$ and let $k\in\N$ such that $n'_{k-1}\leq j< n'_k$. 
	Then, by property ($\mathcal{Q}_1$), $S_jS_{j-1}\cdots S_1=S_j\cdots S_{n'_{k-1}+1}$. 
	Let $X_1= \spann(e_0,e_{k^2})$ and let $X_2=\overline{\spann}(e_n:n\neq 0,k^2)$.
	Observe that $X$ is isomorphic to $X_1 \oplus X_2$.
	Indeed, let $P=e_0^*\otimes e_0+e_{k^2}^*\otimes e_{k^2}$ and let $Q= I-P$.
	Then, $P$ and $Q$ are bounded parallel projections onto $X_1$ and $X_2$ respectively.
	In fact, $\|P\|\leq 2b$. 
	Since $Id=P+Q$, we get that $\|S_j\cdots S_{n'_{k-1}+1}\|\leq \|S_j\cdots S_{n'_{k-1}+1}P\|+\|S_j\cdots S_{n'_{k-1}+1}Q\|$. 
	Thanks to $(\mathcal{Q}_1)$ and $(\mathcal{Q}_2)$, it follows that $S_j\cdots S_{n'_{k-1}+1}Q= \alpha^{-(j-n'_{k-1})} Q$. 
	Thus, $\|S_j\cdots S_{n'_{k-1}+1}Q\|\leq \|Q\|\leq 1+2b$. 
	On the other hand, regarding the operator $S_j\cdots S_{n'_{k-1}+1}P$, we can notice that \[S_j\cdots S_{n'_{k-1}+1}P=(e_1^*+\sigma_je_{k^2}^*)\otimes e_1+\beta_j e_{k^2}^*\otimes e_{k^2},~\text{where } |\sigma_j|, \beta_j\in [0, \alpha^d],\] 
	with which we conclude that $\|S_j\cdots S_{n'_{k-1}+1}P\|\leq b(1+2\alpha^{d})$, a constant independent of $j$. Finally, the proof is finished choosing $M(d)=1+3b+2b\alpha^d$.
\end{proof}
Let us check that $T$ is a non-hypercyclic operator.
Suppose that $z=(e_0,0,\cdots )$ is a cluster point of the orbit of some $w\in \bigoplus_YX$, under the action of $T$. 
Therefore, there exists $(m_k)_k\subset \N$ an increasing sequence such that $(T^{m_k}w)_k$ converges to $z$. 
Hence, the first coordinate of $T^{m_k}w$, which is $S^{-1}_1\cdots S^{-1}_{m_k}w_{m_k}$, tends to $e_0$ as $k$ tends to infinity. 
However, we have that
\begin{align}
	\|w_{m_k}-e_0\| &=\| S_{m_k}\cdots S_1(S_1^{-1}\cdots S_{m_k}^{-1}w_{m_k} -e_0)\|\label{ineq: 3.2} \\
	& \leq M(d)\|S_1^{-1}\cdots S^{-1}_{m_k}w_{m_k}-e_0\|,\notag
\end{align} 
\noindent which implies that $(w_{m_k})_k$ converges to $e_0$. 
This contradicts the fact that $w\in \bigoplus_YX$ because $(\|w_k\|)_k$ does not converge to $0$.
\begin{rem}
	The operator $T$ is not $\delta$-hypercyclic for any $\delta<1/M(d)$. 
	Indeed, let ${w\in \bigoplus_YX}$. 
	Thanks to triangle inequality and replacing the vector $z$ by the vector $\lambda z$ in \eqref{ineq: 3.2} we get that
	\begin{align*} 
		\dfrac{\|\lambda z\|}{M(d)}-\dfrac{\sup\{\|w_k\|:k\in\N\}}{M(d)}&\leq
		\dfrac{\|\lambda e_0\|}{M(d)}-\dfrac{\|w_k\|}{M(d)}\\
		&\leq \|S_1^{-1}\cdots S^{-1}_{k}w_{k}-\lambda e_0\|\\
		&\leq \|T^k w-\lambda z\|, \hspace{0.5cm} \text{for all } k\in \N.
	\end{align*}
	Let us fix $\delta<1/M(d)$. Since $\sup\{\|w_k\|:~k\in\N\}$ is finite, we can choose $\lambda\in\K$ with large modulus to show that $w$ is not a $\delta$-hypercyclic vector of $T$. 
	Finally, since $w$ is an arbitrary vector, $T$ is not a $\delta$-hypercyclic operator.
\end{rem}
\textit{Step 3: $T$ is an $\varepsilon$-hypercyclic operator.} 

\begin{prop}\label{sequences x and z2}
	There exist two sequences $(x^k)_k,~(z^k)_k\subset \bigoplus_YX$ such that 
	
	\begin{enumerate}
		\item[$(1)$] $(x^k)_k$ is dense in $\bigoplus_YX$,
		\item[$(2)$] $\|z^k-x^k\| \leq 2\alpha^{-d}b \|x^k\|,~\text{for all } k\geq 2$, and
		\item[$(3)$] $\|S_{n_k+j}\cdots S_{j+1}z^k_j\|\leq 2^{-k}$ for every $k\geq 2$ and for every $j=0,\cdots ,k-1$.
		\item[$(4)$] For each $k$, there is $N_k\in N$ such that $z^k_j=0$ for all $j\geq N_k$.
	\end{enumerate}
\end{prop}

\begin{proof}
	Let $(x^k)_k\subset \bigoplus_YX\setminus \{0\}$ be a sequence which satisfies the following two properties:
	\begin{enumerate}
		\item $\bigoplus_YX=\overline{\{x^k:k\in\N\}}$.
		\item For each $k\in\N$, $x^k=(x^k_0,\cdots ,x^k_{k-1},0,\cdots )$, where each $x^k_j\in \spann(e_n:n\leq k-1)$.
		%		\item $\|x^k\|\leq 2^k$.
	\end{enumerate}
	Let $k\geq 2$. 
	In order to define $z^k$, let us fix $j<k$ and $l\in\N$ such that $n'_{l-1}\leq j<n'_{l}$. 
	We know that $l<k$.
	Let us define $v_j^k\in X$ by:
	\[
	\alpha^{d+j-n'_{l-1}}v_j^k=\begin{cases}
		e_0^*(x^k_j)e_{k^2}& \text{if }   n'_{l-1}\leq j\leq n'_{l-1}+d,\\
		(e_0^*+\alpha^{j-(n'_{l-1}+d+1)}e^*_{l^2})(x^k_j)e_{k^2} &\text{if } n'_{l-1}+d+1\leq j\leq n'_{l-1}+d+1+\Delta_l=n_l\\
		(e_0^*+\alpha^{\Delta_l-(j-n_{l})}e^*_{l^2})(x^k_j)e_{k^2} &\text{if } n_l+1\leq j\leq n_l+\Delta_l,\\
		e_0^*(x^k_j)e_{k^2} &\text{if } j\geq n_l+\Delta_l+1.\\
	\end{cases}
	\]
	Set $v^k=(v^k_0,v^k_1,\cdots ,v^k_{k-1},0,\cdots )$ and $z^k=x^k+v^k$. 
	Observe that $\|v^k_j\| \leq 2\alpha^{-d}b\|x^k_j\| $, for all $j\in\{0,1,\cdots ,k-1\}$. 
	Since the space $\bigoplus_YX$ is constructed with a $1$-unconditional basis of $Y$, we conclude that $\|z^k-x^k\|\leq 2\alpha^{-d}b\|x^k\|$.
	So, it only remains to prove property $(3)$.
	Let $k\geq 2$ and $j\in \{0,\cdots ,k-1\}$. 
	For the sake of brevity, let us set $c_j^k\in \K$ by $v_j^k=c_j^ke_{k^2}$.
	Let $l\in\N$ such that $n'_{l-1}\leq j<n'_{l}$. Then, we get
	\begin{align*}
		S_{n_k+j}\cdots S_{j+1}z^k_j &= S_{n_k+j}\cdots S_{1}(S_{1}^{-1}\cdots S_{j}^{-1}z^k_j)\\
		&= S_{n_k+j}\cdots S_{n'_{k-1}+1}(S_{n'_{l-1}+1}^{-1}\cdots S_{j}^{-1}(x_j^k +v_j^k))\\
		&=S_{n_k+j}\cdots S_{n'_{k-1}+1}(S_{n'_{l-1}+1}^{-1}\cdots S_{j}^{-1}(x_j^k)+ \alpha^{j-n'_{l-1}}v_j^k),\\
		&=S_{n_k+j}\cdots S_{n'_{k-1}+1}(S_{n'_{l-1}+1}^{-1}\cdots S_{j}^{-1}(x_j^k))+\alpha^{j-n'_{l-1}}(-\alpha^{d}c^k_je_0+\alpha^{d-\Delta_k+j}c_j^ke_{k^2}).
	\end{align*}
	where the second equality comes from ($\mathcal{Q}_1$), the third one is due to ($\mathcal{Q}_2$) and the fact that $l<k$ and in the last line we have assumed that $\Delta_k$ is bigger than $k$. 
	To continue, let us set the vector $h:=S_{n'_{l-1}+1}^{-1}\cdots S_{j}^{-1}(x_j^k)$ and the operators $P=e^*_0\otimes e_0$ and $Q=I-P$. 
	Then, since the operators $\{S_j:j\geq 1\}$ are upper-triangular with respect to the $M$-basis $(e_n)_n$, we conclude that $Qh\in\spann\{e_n:0<n<k\}$. Thus, we get
	\begin{align*}
		S_{n_k+j}\cdots S_{j+1}z^k_j &=S_{n_k+j}\cdots S_{n'_{k-1}+1}(Ph+Qh)+\alpha^{j-n'_{l-1}}(-\alpha^{d}c^k_je_0+\alpha^{d-\Delta_k+j}c_j^ke_{k^2}),\\
		&=[Ph-\alpha^{d+j-n'_{l-1}}c^k_je_0]+\alpha^{-(n_k+j-n'_{k-1})}Qh+\alpha^{j-n'_{l-1}+d-\Delta_k+j}c_j^ke_{k^2},
	\end{align*}
	where in the second line we have used property ($\mathcal{Q}_0$) and that the operator $S_j$ restricted to $\spann(e_n:0<n<k)$ is equal to $\alpha^{-1} Id$ for all $j \in [n'_{k-1}+1, n'_k-1]$. Since $n_k+j-n'_{k-1}=j+\Delta_k+d+1$ and $\|Qh\|$ does not depend on $\Delta_k$, because $l<k$, the third term in the last expression tends to $0$ as $\Delta_k$ tends to infinity.
	Also, since $l<k$ and $|c_j^k|=\|v_j^k\| \leq 2\alpha^{-d}b\|x^k_j\|$ does not depend on $\Delta_k$, the fourth term in the last expression tends to $0$ as $\Delta_k$ tends to infinity.
	On the other hand, the coefficients $c_j^k$ were chosen to cancel the expression enclosed in square brackets. 
	Finally, if we choose $\Delta_k$ large enough (with $\Delta_k>k$), we can ensure that $\|S_{n_k+j}\cdots S_{j+1}z^k_j\|\leq 2^{-k}$.
\end{proof}

\begin{proof}[Proof of Theorem \ref{ehypercyclic2}]
	
	We already know that $T$ is a bounded non-hypercyclic operator on $\bigoplus_YX$. 
	Let us show that $T$ is $\varepsilon$-hypercyclic, using the $\varepsilon$-Hypercyclicity Criterion, Theorem \ref{criterion}. 
	Let $(x^k)_k$ and $(z^k)_k$ be sequences given by Proposition \ref{sequences x and z2}.
	Let us set 
	\[\D_1:=\{(y_i)_i\in \bigoplus_YX:~\exists N\in \N, y_i=0,~\forall i\geq N\},\]
	\noindent which is in $\bigoplus_YX$. 
	Let $\D_2:=\{z^k\in \bigoplus_YX:~k\geq 2\}$. 
	Let $w\in \bigoplus_YX$ be a vector different from $0$ and let $(x^{m_k})_k$ be a subsequence of $(x^k)_k$ which converges to $w$.
	Let $\rho>2\alpha^{-d}b$.
	We claim that, for $k$ large enough, $z^{m_k}\in B(w,\rho\|w\|)$.
	In fact, applying Proposition \ref{sequences x and z2} $(2)$ we obtain that
	\begin{align*}
		\|w-z^{m_k}\|\leq \|w-x^{m_k}\|+\|x^{m_k}-z^{m_k}\|\leq \|w-x^{m_k}\| + 2\alpha^{-d}b\|x^{m_k}\|,
	\end{align*}
	Since $\rho>2\alpha^{-d}b$ and $(x^{m_k})_k$ converges to $w$, the claim is proved. 
	Thus, there are infinitely many $k\in\N$ such that $z^k\in B(w,\rho\|w\|)$.
	Let $(n(k))=n_k$ be the sequence constructed in the first part. 
	Now, we check the three hypotheses of the $\varepsilon$-Hypercyclicity Criterion. 
	Let us define the map $U:\D_1\to\D_1$ as the formal right inverse of $T$. i.e. $U$ is defined by
	\[U(y_i)_i=T^{-1}(y_i)_i=(0, S_1y_0, S_2y_1,\cdots ),~\forall (y_i)_i\in \D_1.\]
	Let $U_{n(k)}:=T^{-n(k)}$. 
	With this, hypothesis $(1)$ and $(3)$ are straightforward. 
	Indeed, let $y\in \D_1$. Since $(n(k))_k$ tends to infinity and $T$ is a backward shift, we have that $T^{n(k)}y=0$ for $k$ large enough. 
	Hypothesis $(3)$ follows from the formula $T^{n(k)}U_{n(k)}= Id$, which is valid in $\D_2$ thanks to Proposition~\eqref{sequences x and z2}~$(d)$.
	Finally, hypothesis $(2)$ is implied by Proposition~\ref{sequences x and z2}~$(c)$. 
	Indeed, let $k\geq 2$.
	By triangle inequality we have that
	\begin{align*}
		\|U_{n(k)}z^k\|&\leq \sum_{j=0}^{k-1}\|U_{n(k)}(0,\cdots ,0,z^k_j,0,\cdots )\|\leq k2^{-k},
	\end{align*}
	\noindent expression which tends to $0$ as $k$ tends to $\infty$. 
	Hence, $T$ is a $\rho'$-hypercyclic operator for any $\rho'>\rho$. 
	Finally, since $\rho$ can be chosen arbitrary close to $2\alpha^{-d}b$, and then $\rho<\varepsilon$, we finally get that $T$ is an $\varepsilon$-hypercyclic operator.
\end{proof}

\begin{rem}\label{sequence nk}
    Notice that the sequence $(\Delta_k)_k$ used in the construction of the $\varepsilon$-hypercyclic operator $T$ can be replaced by any sequence of integers $(\Delta'_k)_k$ so that $\Delta_k\leq \Delta'_k$, for all $k\in\N$.
	Observe that the operator constructed in Theorem \ref{ehypercyclic2} satisfies the $\varepsilon$-Hypercyclicity Criterion associated to the sequence $(n_k)_k$, where \[n_k=(2k-1)(d+1)+\Delta_k+2\sum_{j=1}^{k-1}\Delta_j , ~\forall k\geq 1.\]
	
\end{rem}
%%%

In order to extend further our result we recall that there exist hypercyclic operators in each separable Banach space, see \cite{A} and \cite{Be}. Further, in \cite{LM}, Le\'on-Saavedra and Montes-Rodr\'iguez showed that the operator constructed in \cite{Be} satisfy the Hypercyclicity Criterion.

\begin{thm}\label{Theorem product spaces}
	Let $X$ be a separable Banach space. Assume that $X$ admits an infinite dimensional complemented subspace $V$ of the form $V=\bigoplus_YZ$, where $Y$ and $Z$ satisfy the assumptions of Theorem \ref{ehypercyclic2}. 
	Then $X$ admits an $\varepsilon$-hypercyclic operator which is not hypercyclic.
\end{thm}
Observe that Theorem \ref{product space corollary} is just Theorem \ref{Theorem product spaces} whenever the space $V$ is either $c_0(\N)$ or $\ell^p(\N)$, for $p\in[1,\infty)$.
\begin{proof}
	Let $\varepsilon>0$.
Let $V=\bigoplus _YZ$ be the complemented subspace given by the statement.
Let $W$ be a topological complement of $V$ on $X$. 
Without loss of generality, we assume that $W$ is infinite dimensional. 
Otherwise, considering $(f_n)_n$ as the basis of $Y$ used in the construction of $V$, we replace $Y$ for $\overline{\spann}(f_n:n\geq 1)$ and $W$ for $W\oplus Z$, which is infinite dimensional.
Let us consider $T$ be any bounded hypercyclic operator on $W$ that satisfies the Hypercyclicity Criterion. 
Let $(n_k)_k$ be a sequence of integers provided by the Hypercyclicity Criterion for $T$. 
By Theorem \ref{ehypercyclic2}, there is an $\varepsilon$-hypercyclic operator $S$ on $V$ which is not hypercyclic.
Moreover, by Remark \ref{sequence nk}, we can chose $S$ such that satisfies the $\varepsilon$-Hypercyclicity Criterion for a sequence $(m_k)_k$ of the form \[m_k=2k(d+1)+2\sum_{j=1}^{k-1}\Delta_j +\Delta_k, ~\hspace{0.5cm} \text{for all } k\geq 1.\] 
Since, for each $j\in\N$, we can chose $\Delta_j$ as large as we want, we can (and shall) assume that the sequence $(m_k)_k$ is a subsequence of $(n_k)_k$. 
Therefore, since $T$ also satisfies the Hypercyclicity Criterion for the sequence $(m_k)_k$, we can apply Proposition \ref{product space} to deduce that $S\oplus T$ is $\delta$-hypercyclic on $V\oplus W$, for all $\delta>\varepsilon$.
However, $S\oplus T$ is not hypercyclic.
Indeed, notice that $V$ and $W$ are complemented spaces and both are invariant for $S\oplus T$. 
If $S\oplus T$ were hypercyclic, then both restriction, $S\oplus T|_V$ and $S\oplus T|_W$ would be hypercyclic as well. 
However, $S\oplus T|_V=S$, which is not hypercyclic.
\end{proof}

\section{A remark on the $\varepsilon$-Hypercyclicity Criterion}

One of the main differences between the proposed $\varepsilon$-Hypercyclicity Criterion and the Hypercyclicity Criterion is the necessity of an enumeration of the set $\mathcal{D}_2$. In fact, in the literature we can find several criteria, having a structure similar to the Hypercyclicity Criterion, in which the corresponding set $\mathcal{D}_2$ is not necessarily enumerated. For instance, regarding the criteria for supercyclicity, cyclicity or frequent hypercyciclity stated in \cite[Theorem 1.14, Exercise 1.4 and Theorem 6.18]{BM} respectively, the conditions on every point of $\mathcal{D}_2$ is identical. However, the next result says that we cannot naively avoid this technicality.

\begin{prop}\label{almost criterion}
	Let $X$ be an infinite dimensional separable Banach space, let $T$ be a bounded operator on $X$ and let $\varepsilon\in(0,1)$. Let $\D_1$ be a dense set on $X$. Let $\D_2$ be a subset of $X$ such that $\D_2\cap B(x,\varepsilon\|x\|)$ is nonempty for all $x\in X\setminus \{0\}$. 
	Let $(n(k))_k\subset \N$ be an increasing sequence and let $S_{n(k)}:\D_2\to X$ be a sequence of maps such that:
	\begin{enumerate}
		\item $\lim_{k\to\infty}\|T^{n(k)}x\|= 0$ for all $x\in \D_1$,
		\item $ \lim_{k\to\infty}\|S_{n(k)}y\|=0$, for all $y\in \D_2$,
		\item $ \lim_{k\to\infty}\|T^{n(k)}S_{n(k)}y-y\| = 0$ for all $y\in \D_2$.
	\end{enumerate}
	Then, $T$ satisfies the Hypercyclicity criterion.
\end{prop}
Before proceeding with the proof, we recall that an operator $T$ on $X$ is cyclic if there exists a vector $x\in X$ such that $\spann(\textup{Orb}_T(x))$ is dense in $X$.
\begin{proof} 
	The proof follows by showing that $T\oplus T$ is a cyclic operator on $X\times X$. 
	Indeed, if $T\oplus T$ is cyclic, then $T\oplus T$ is hypercyclic by \cite[Proposition 4.1]{Gr} and, finally, $T$ satisfies the Hypercyclicity criterion by \cite[Theorem 2.3]{BP}. 
	Since the argument is analogous to the one presented in the proof of Proposition \ref{product space}, we present only a sketch of the proof. 
	First, we fix a sequence $(v_k)\subset X\setminus \{0\}$ which is dense in $X$. Let us consider a countable partition of $\N$ given by infinite countable set, namely, $\N=\bigcup_j \N_j$. 
	By Remark \ref{remark limsup}, we can construct a vector $z_1\in X$ and an increasing sequence $(k(i))_i\subset \N$ such that:
	\[\limsup_{i\in\N_j,~i\to\infty}\|T^{n(k(i))}z_1-v_j\|\leq \varepsilon\|v_j\|,~ \forall~j\in\N.\]
	Now, we construct a vector $z_2$ adapted to $z_1$ in the following sense:
	\[\liminf_{i\in\N_j,~i\to\infty}\|T^{n(k(i))}z_2-x\|\leq \varepsilon\|x\|,~ \forall x\in X\setminus\{0\}, \forall~j\in\N.\]
	Finally, $(z_1,z_2)$ is an $\varepsilon$-hypercyclic vector of $T\oplus T$ on $X\times X$ whenever this space is endowed with the norm of the maximum. 
	Hence, by Proposition \ref{ehyp c} $(z_1,z_2)$ is a cyclic vector of $T\oplus T$, and so $T$ satisfies the Hypercyclicity Criterion.
\end{proof}
%\begin{ques}
%	The following are natural questions arising from our results.
%	\begin{enumerate}
%		\item Is the $\varepsilon$-Hypercyclicity Criterion a necessary condition for $\delta$-Hypercyclicity, for some $\delta\geq\varepsilon$?
%		\item Does there exist a hypercyclic operator which does not satisfy the $\varepsilon$-Hypercyclicity Criterion, for some $\varepsilon>0$?
%	\end{enumerate}
%\end{ques}
%%%
%%%
%%%

\section{Elementary results}

\begin{prop}\label{isomorphism}
	Let $(X_1,\|\cdot\|_1)$ and $(X_2,\|\cdot\|_2)$ be two isomorphic Banach spaces. Assume that, for each $\varepsilon>0$, $X_1$ admits an $\varepsilon$-hypercyclic operator which is not hypercyclic. 
	Then $X_2$ enjoys the same property.
\end{prop}

\begin{proof}
	Let $T\in\mathcal{L}(X_1,X_2)$ be an isomorphism between $X_1$ and $X_2$. 
	Let $\varepsilon\in (0,1)$ and let $S$ be an $\varepsilon$-hypercyclic operator on $X_1$ which is not hypercyclic. 
	We claim that $TST^{-1}$ is a $\|T\|\|T^{-1}\|\varepsilon$-hypercyclic but not hypercyclic operator on $X_2$. 
	Indeed, let $x\in X_1$ be an $\varepsilon$-hypercyclic vector of $S$. 
	Let $y\in X_2$ and $n\in \N$ be an integer such that $\|S^nx-T^{-1}y\|_1 \leq \varepsilon\|T^{-1}y\|_1$.  Now, we can observer that
	\[\|TS^nT^{-1}(Tx)-y\|_2\leq \|T\|\|S^nx-T^{-1}y\|_1 \leq \|T\|\|T^{-1}\|\varepsilon\|y\|_2,\]
	concluding that $Tx$ is an $\|T\|\|T^{-1}\|\varepsilon$-hypercyclic vector of $TST^{-1}$. Finally, $TST^{-1}$ cannot be hypercyclic since this property is preserved under conjugacy.
\end{proof}

\begin{prop}\label{ehyp c}
	Let $T$ be an $\varepsilon$-hypercyclic operator on $X$, with $\varepsilon\in (0,1)$. Then $T$ is a cyclic operator.
\end{prop}

\begin{proof}
	It is a direct consequence of the following well-known result. 
	Let $Y$ be a closed subspace of $X$ different from $X$, then for any $\delta>0$ there exists a unitary vector $z\in X\setminus Y$ such that $\dist(z,Y)\geq 1-\delta$.
	Indeed, assume by contradiction that $T$ is a non-cyclic operator and let $Y=\overline{\textup{span}}(\textup{Orb}_T(x))$, where $x$ is an $\varepsilon$-hypercyclic vector of $T$. 
	Let $\delta\in(\varepsilon,1)$ and let $z\in X\setminus Y$ be a unitary vector such that $\textup{dist}(z,Y)> \delta$. 
	Therefore $B(z,\delta)\cap Y=\emptyset$.
	Hence, $x$ is not a $\delta$-hypercyclic vector, and thus, $x$ cannot be an $\varepsilon$-hypercyclic vector, which is a contradiction.
\end{proof}

\section{Acknowledgements}
The author is grateful to Robert Deville for introducing him into the results of Badea, Grivaux, Müller and Bayart, and for fruitful discussions. This work was supported by ANID-PFCHA/Doctorado Nacional/2018-21181905 and by CMM (UMI CNRS 2807), Basal grant: AFB170001.

\bibliographystyle{amsplain}

\end{document}